\documentclass[preprint]{amsart}   
\usepackage[utf8]{inputenc}
\usepackage[english]{babel}
\usepackage{graphicx}
\usepackage{amsmath}
\usepackage{amsfonts}
\usepackage{amssymb}
\usepackage[numbers]{natbib}
\usepackage[none]{hyphenat}
\usepackage{amssymb}
\usepackage{enumerate}
\usepackage[mathscr]{euscript}
\newtheorem{theorem}{Theorem}
\newtheorem{lemma}{Lemma}
\newtheorem{example}{Example}
\newtheorem{defn}{Definition}
\newtheorem{cor}{Corollary}
\usepackage{ragged2e}

\author{Madhurima Datta}
\address[M.~Datta]{Department of Mathematics, Indian Institute of Technology Kharagpur, India- 721302.}
\email[M.~Datta]{madhurima92.datta@iitkgp.ac.in}

\author{Nitin Gupta}
\address[N.Gupta]{Department of Mathematics, Indian Institute of Technology Kharagpur, India- 721302.}
\email[N.Gupta]{nitin.gupta@maths.iitkgp.ernet.in}

\begin{document}

\title{Stochastic comparison study of the general family of PHR and PRHR distributions}


\maketitle

\begin{abstract}
In this paper, we have obtained conditions on parameters that result in dispersive ordering and star ordering among two unequal sets of random variables from Proportional hazard rate and Proportional reversed hazard rate family of distributions. The n-independent random variables under observation belonging to a multiple-outlier model are given as corollary. The conditions obtained involve simple inequalities among the parameters and class of life distribution corresponding to ageing. Some stochastic ordering results for sample minimum and maximum with dependency based on Archimedean copula has also been studied by varying the location parameters.
\end{abstract}


\justifying
\section{Introduction}
\label{intro}
The series and parallel systems are the most frequent and maximum encountered systems in nature. These systems are statistically referred to as the minimum and the maximum order statistics respectively. Let $X_1,X_2,\ldots, X_n$ be n independent and non-identical random variables from a particular population. Then arranging the random variables according to their magnitude or strength we observe that $X_{1:n}\leq X_{2:n} \leq \ldots \leq X_{n:n}$, where $X_{k:n}$ is known as the k-th order statistic. $X_{k:n}$ represents the lifetime of a (n-k+1)-out-of-n system. In this paper, we focus only on the minimum and maximum order statistic. A great deal of literature is available on the stochastic relationship among the order statistics of various distributions. \par
In particular, our problem deals with the Proportional hazard rate (PHR) model and proportional reversed hazard rate (PRHR) model. We consider $X_1,X_2,\ldots, X_n$ as independent random variables each following the PHR model. The reliability or survival probability of $X_i$ is:
$$P(X_i > x) = \overline{F}_i(x) = [\overline{F}_0(x)]^{\lambda_i} , i=1,2,\ldots,n$$
where $\lambda_i$ is the proportionality parameter. Let $X_0$ be a baseline random variable with baseline distribution $F_0(x)$ and baseline survival function $\overline{F}_0(x)=1-F_0(x)$. Exponential, Weibull, Pareto, Lomax, Kumaraswamy's distributions are examples of PHR model distribution. \cite{1} pioneered the study of stochastic ordering (details about stochastic ordering given in the next section) for k-out-of-n systems which included usual stochastic ordering results for PHR model. \cite{2} studied dispersive and star ordering for general distributions in detail. Later on, many researchers have continued the study and have found many results for PHR model. Some discussed results for exponential distribution and later on extended them for PHR models, this was possible as the random variable corresponding to the cumulative hazard rate function of a PHR family of distribution follows exponential distribution with the proportionality constant as the parameter i.e., if $X$ follows $[\overline{F}(x)]^{\lambda}$, then the cumulative hazard rate function follows Exp($\lambda$) distribution.
\cite{3} demonstrated dispersive ordering between the maximum order statistics of two PHR populations. \cite{4}, \cite{7} observed dispersive ordering between the $2^{nd}$ order statistics (also known as fail-safe systems) from two different populations and derived bounds on the parameters. \cite{5} reviewed the works done for dispersive ordering in parallel systems with components following PHR model. A comprehensive review of the various ordering between the order statistics for random variables belonging from the PHR model has been done by \cite{17}. Recently \cite{own} observed stochastic ordering for Kumaraswamy's and Frechet distributed components.\par
Let $X_1,X_2,\ldots,X_n$ and $Y_1,Y_2,\ldots,Y_n$ be n-independent PHR samples having the same baseline distribution but the parameter vectors are given by \\$(\underbrace{\alpha_1,\ldots, \alpha_1}_{p},\underbrace{\alpha_2,\ldots,\alpha_2}_q)$ and $(\underbrace{\beta_1,\ldots,\beta_1}_p,\underbrace{\beta_2,\ldots,\beta_2}_q)$ respectively, where $p+q = n$. Such an arrangement is described as the multiple-outlier model. \cite{10}, \cite{8} discussed hazard rate and likelihood ratio ordering for parallel systems with multiple-outlier PHR model. For a similar model, \cite{9} expressed conditions on the distribution function for the dispersive ordering of k-th order statistics where the parameter vectors follow majorization relation. \cite{11} found the necessary and sufficient conditions for the hazard rate ordering among the second order statistics. \cite{resilience} examined the stochastic comparison between series and parallel systems where the component lifetimes are dependent, heterogeneous and resilience scaled. \cite{min}, \cite{max} found several conditions for stochastic ordering of minimum and maximum order statistics from a location-scale family of distributions. \cite{range} observed stochastic ordering between the sample ranges where component lifetimes (number of components are different) are independent and follows multiple-outlier exponential distribution and PHR  models. \par
In contrast to the PHR model, proportional reversed hazard rate (PRHR) model was developed. Let $X_i$ follows PRHR model then the distribution function of $X_i$ is given by 
$$P(X_i < x) = F_i(x) = [F_0(x)]^{\theta_i}$$
where $\theta_i$ is the proportionality constant. Some known examples of PRHR model are exponentiated Weibull, exponentiated exponential, exponentiated Gamma, etc. A random variable has decreasing reversed hazard rate(DRHR) if and only if the distribution function is log-concave. It is known that there exists no distribution which is log convex or increasing reversed hazard rate(IRHR) over the entire domain $[0,\infty)$. An IRHR distribution can be constructed if the domain is taken as $(-\infty, \alpha)$ for some finite $\alpha$ (see \cite{12}). \cite{13} observed dispersive ordering for the series systems with components following the PRHR model.\par
In our present study, we have considered two sets of independent PHR and PRHR models where the baseline distribution of both the sets are different and the sample sizes are also different  i.e., the first set of random variables $X_i \sim \overline{F}_i(x) = (\overline{F}_{0}(x))^{\alpha_i}$ for $i=1, 2,\ldots, n_1$ and the second set $Y_i \sim \overline{G}_i(x) = (\overline{G}_{0}(x))^{\beta_i}$ for $i=1, 2,\ldots, n_2$. We have studied dispersive ordering for this model and star ordering has been observed when the baseline distributions remain the same. A similar kind of model has been studied for the PRHR model. Only the survival functions in the above model changes to distribution functions for the PRHR model. While in the other model,  $X_1,\ldots,X_{p_1}$ has survival function $[\overline{F}(x)]^{\alpha_i}$ and $X_{p_1+1},\ldots,X_{n_1}$ has survival function $[\overline{G}
(x)]^{\alpha_i}$. And  $Y_1,\ldots,Y_{p_2}$ has survival function as $[\overline{F}(x)]^{\beta_i}$ whereas the components $Y_{p_2+1},\ldots,Y_{n_2}$ has survival function $[\overline{G}(x)]^{\beta_i}$. Hazard rate ordering  for sample minimum exists for such model, analogously reversed hazard rate ordering for sample maximums exist for PRHR model. A reversed hazard rate ordering for sample maximum (with equal sample sizes) for Pareto distributed random variables has been observed when only the shape parameter varies. \par
Lastly, some results on random variables belonging from the dependent model have been studied where the Archimedean copula has been considered as the survival copula. The studies include the results obtained when the location parameter is varied along with a comparison between two generating functions (super-additive property) and usual stochastic ordering among baseline distributions have been considered.\\
The paper has been constructed as follows: Section 2 includes all the definitions used in the paper, Section 3 contains dispersive, hazard rate and reversed hazard rate ordering among PHR and PRHR model with unequal sample sizes. Section 4 has one result for star ordering between sample minimums with unequal sample sizes. Section 5 considers the results for the dependent model. The various well-known lemmas that have been used in proving the results are attached in the appendix.

\section{Definitions}
Let $X$ and $Y$ be two absolutely continuous random variables with distribution functions $F(x)$ and $G(x)$; reliability functions as $\overline{F}(x)$ and $\overline{G}(x)$; probability density functions as $f(x)$ and $g(x)$; hazard rate functions as $r(x) = \dfrac{f(x)}{\overline{F}(x)}$ and $s(x) = \dfrac{g(x)}{\overline{G}(x)}$; reversed hazard rate functions as $\tilde{r}(x) = \dfrac{f(x)}{F(x)}$ and $\tilde{s}(x) = \dfrac{g(x)}{G(x)}$. Here $F^{-1}$ and $G^{-1}$ are the right continuous quantiles of $X$ and $Y$ respectively. A real valued function $\psi$ is super-additive when $\psi(x_1+x_2) \geq \psi(x_1)+\psi(x_2)$ for all $x_1,x_2 \in Domain(\psi)$. This concept is valid even when the summation is over n-variables. For details on the above definitions we refer the reader to \cite{14}. Next we discuss some of the various stochastic orders available in literature. We refer the reader to \cite{15} for the detail of these orderings.

\begin{defn}
We say $X$ is smaller than $Y$ in 
\begin{enumerate}[(a)]
\item usual stochastic order ($X \leq_{st} Y$) if and only if $\overline{F}(x) \leq \overline{G}(x)$ $\forall$ $x \in (-\infty, \infty)$.

\item hazard rate order $(X \leq_{hr} Y)$ if
 $r(x) \geq s(x), x \in \mathbb{R}$.
Equivalently, if $ \dfrac{\overline{G}(x)}{\overline{F}(x)}$ is increasing in $x$ over the union of the supports of $X$ and $Y$.

\item reversed hazard rate order $(X \leq_{rh} Y)$ if $\tilde{r}(x) \leq \tilde{s}(x), x \in \mathbb{R}$.
Equivalently, if $ \dfrac{G(x)}{F(x)}$ is increasing in $x$ over the union of the supports of $X$ and $Y$.

 \item likelihood ratio order $(X \leq_{lr} Y)$ if $\dfrac{g(x)}{f(x)}$ is increasing in $x$ over the union of the supports of $X$ and $Y$. 
 
 \item dispersive order $(X \leq_{disp} Y)$ if $$F^{-1}(\alpha_2) - F^{-1}(\alpha_1) \leq G^{-1}(\alpha_2) - G^{-1}(\alpha_1) \text{ whenever } 0 < \alpha_1 \leq \alpha_2 < 1.$$
Equivalently, $(X \leq_{disp} Y)$ if and only if
$$G^{-1}(\alpha) - F^{-1}(\alpha) \text{ increases in } \alpha \in (0,1).$$

\item star order $(X \leq_{*} Y)$ if $\dfrac{G^{-1}(t)}{F^{-1}(t)}$ increases in $t \in (0,1)$. 

\end{enumerate}
\end{defn}

\begin{defn}
\normalfont
\textbf{Majorization}: \\

Let $\underline{a} = (a_1, \ldots, a_n)$ and $\underline{b} = (b_1, \ldots, b_n)$ be two real valued vectors then
\begin{enumerate}
\item $\underline{a}$ is majorized by $\underline{b}$   ( $\underline{a} \prec \underline{b}$ ) if
\begin{equation}
 \sum_{i=1}^{n}a_{i:n} = \sum_{i=1}^{n}b_{i:n} ~\text{and}~ \sum_{i=1}^{k}a_{i:n} \geq \sum_{i=1}^{k}b_{i:n} ~\forall~k = 1,\ldots, n-1;
 \end{equation}
\item $\underline{a}$ is weakly submajorized by $\underline{b}$ ( $\underline{a} \prec_{w} \underline{b}$ ) if
\begin{equation}
  \sum_{i=1}^k a_{n-i+1:n} \leq \sum_{i=1}^k b_{n-i+1:n} ~\forall~ k = 1, \ldots, n;
\end{equation}
\item $\underline{a}$ is weakly supermajorized by $\underline{b}$ ( $\underline{a} \prec^{w} \underline{b}$ ) if
\begin{equation}
\sum_{i=1}^k a_{i:n} \geq \sum_{i=1}^k b_{i:n} ~\forall~ k = 1, \ldots, n;
\end{equation}

 where $a_{1:n} \leq \ldots \leq a_{n:n} ~( b_{1:n} \leq \ldots \leq b_{n:n}) $ is the increasing arrangement of $a_1, \ldots, a_n ~(b_1, \ldots, b_n)$.
\end{enumerate}
For $\underline{a}$ and $\underline{b}$, we have $\underline{a} \prec^{w} \underline{b} \Leftarrow \underline{a} \prec \underline{b} \Rightarrow \underline{a} \prec_{w} \underline{b}$.
\end{defn}

\begin{defn}
\normalfont \textbf{Schur-convexity (Schur-concavity)}: A real valued function $\psi$ defined on a subset of $\mathbb{R}^n$ is \emph{Schur-convex (Schur-concave)} if
\begin{equation}
\underline{a} \prec \underline{b} \Rightarrow \psi(\underline{a}) ~\leq (\geq)~ \psi(\underline{b}),
\end{equation}
where $\underline{a} = (a_1, \ldots, a_n)$ and $\underline{b} = (b_1, \ldots, b_n)$ are two real valued vectors.
\end{defn}
Throughout the paper, the notation $a \overset{\text{sgn}}{=} b$ has been used to represent sign of $a$ is same as $b$. The results and lemmas that are used in obtaining the proofs are mentioned in the Appendix.

\section{Dispersive ordering results for unequal sample sizes}

In this section we compare minimum and maxmimum order statistics arising from general Proportional hazard rate and proportional reversed hazard rate models. As a corollary some results for multiple-outlier models has also been included here. The multiple-outlier model has been explained in \cite{10,8} as an independent set of random variables $X_1,X_2,\ldots,X_n$, where $F_{X_i}=F_X$ for $i=1,\ldots,p$ and $F_{X_i}=F_Y$ for $i=p+1,\ldots,n$, necessarily $1 \leq p <n$. When the value of $p=n-1$, this becomes a single-outlier model. Earlier many researchers have studied various results for the comparison of order statistics from multiple-outlier models. \cite{9} considered the following model 
$$(X_1,X_2,\ldots,X_n) \sim (\underbrace{(\overline{F}(x))^{\alpha_1}, \ldots, (\overline{F}(x))^{\alpha_1}}_{p}, \underbrace{(\overline{F}(x))^{\alpha_2}, \ldots, (\overline{F}(x))^{\alpha_2}}_{q})$$
and 
$$(Y_1,Y_2,\ldots,Y_n) \sim (\underbrace{(\overline{F}(x))^{\alpha_1^*}, \ldots, (\overline{F}(x))^{\alpha_1^*}}_{p}, \underbrace{(\overline{F}(x))^{\alpha_2^*}, \ldots, (\overline{F}(x))^{\alpha_2^*}}_{q}).$$
They observed star and dispersive ordering for the k$^{th}$ order statistic by imposing majorization properties over the parameters. In the following paper they primarily discussed hazard rate ordering for exponentially distributed components and derived similar hazard rate ordering results for maximum order statistic with some additional conditions over the parameters. Under the same conditions \cite{11} observed hazard rate ordering for second order statistic. Moreover they found hazard rate orderings when the number of components and number of outliers were different. Whereas \cite{r4} studied maximum order statistic for PHR model (survival function of $X_i$ is $\overline{F}_{X_i}(x) = (\overline{F}(x))^{\alpha_i} $ for $i=1,\ldots,n$) such that the distribution function of $\max_{i \in P} X_i$, $P \subset \{1,2,\ldots,n\}$ is 
$$F_{\max}(x)= Q_P(F(x))$$
where $Q_P$ is a distortion function (continuous and increasing in [0,1], also Q(0)=0, Q(1)=1) and it depends on the underlying copula and the proportionality parameters. Few results were observed for different subsets of $\{1,2,\ldots,n\}$.
We have considered various model in our study which includes models where the 
baseline distributions are same but the shape parameter varies, the baseline distributions are different and the shape parameters are also different. Several researchers have studied multiple-outlier models extensively as it helps in dealing with outliers. Recently, \cite{var} studied some results where the  n-component lifetimes of both the systems are dependent with multiple-outlier proportional hazard rates. In our present study we first observe results for  series systems where the component lifetimes are independent and follows different proportional hazard rates (the number of components in both the systems are not necessarily same) and the results for multiple-outlier models can be derived subsequently.  \\

The following theorem has been observed for series systems with components following multiple-outlier PHR family such that the baseline distribution for both the sets are different.
\begin{theorem}\label{newthm1}
	Let $X_1,X_2, \ldots, X_{n_1}$ be a set of n-independent random variables each belonging from a particular PHR family with parameters $(\alpha_1, \alpha_2,\ldots, \alpha_{n_1})$. We assume that $X_i \sim \overline{F}_i(x) = (\overline{F}_{0}(x))^{\alpha_i}$ for $i=1, 2,\ldots, n_1$. Also $Y_1, Y_2,\ldots, Y_{n_2}$ be another set of n-independent random variables each following PHR family of distributions with a different distribution function and the parameter set is $(\beta_1, \beta_2 \ldots, \beta_{n_2})$. Let $Y_i \sim \overline{G}_i(x) = (\overline{G}_{0}(x))^{\beta_i}$ for $i=1, 2,\ldots, {n_2}$. Under the assumption that $\displaystyle\sum_{i=1}^{n_2} \beta_i \geq \displaystyle\sum_{i=1}^{n_1} \alpha_i$, the baseline distribution function $F_{0}$ is DFR,  and $G_{0} \leq_{hr} F_{0}$ then $Y_{1:n_2} \leq_{disp} X_{1:n_1}$.
\end{theorem}
\begin{proof}
	The distribution function of $X_{1:n_1}$ and $Y_{1:n_2}$ are 
	\begin{equation}
	\overline{F}_{1:n_1}(x) = \left(\overline{F}_{0}(x)\right)^{\displaystyle\sum_{i=1}^{n_1} \alpha_i},
	\end{equation}
	and,
	\begin{equation}
	\overline{G}_{1:n_2}(x) = \left(\overline{G}_{0}(x)\right)^{\displaystyle\sum_{i=1}^{n_2} \beta_i}
	\end{equation}
	respectively. For simplicity we replace $\displaystyle\sum_{i=1}^{n_1} \alpha_i$ by $ \alpha$ and $\displaystyle\sum_{i=1}^{n_2} \beta_i$ by $\beta$.
	Let
	\begin{align*}
	\psi_1(y) &= F_{1:n_1}^{-1}(y) - G_{1:n_2}^{-1}(y)\\
	&= \overline{F}_0^{-1}\left((1-y)^{1/{\alpha}}\right) - \overline{G}_0^{-1}\left((1-y)^{1/{\beta}}\right)
	\end{align*}
	We are required to prove  $Y_{1:n_2} \leq_{disp} X_{1:n_1}$, i.e., $\psi_1(y)$  is increasing in  $y \in (0,1)$. Hence $Y_{1:n_2} \leq_{disp} X_{1:n_1}$ if and only if 
	$\phi_1(t) = \overline{F}_0^{-1}(t) - \overline{G}_0^{-1}\left(t^{\dfrac{\alpha}{\beta}}\right) \\
	\text{ is decreasing in } t \in (0,1),$ where $t=(1-y)^{1/\alpha}$. Note that
	\begin{align}
	\phi_1^{'}(t) &= -\dfrac{1}{f_0(\overline{F}_0^{-1}(t))} + \dfrac{\alpha}{\beta}\dfrac{t^{\dfrac{\alpha}{\beta} - 1}}{g_0\left(\overline{G}_0^{-1}\left(t^{\dfrac{\alpha}{\beta}}\right)\right)}.
	\end{align}
	We need to show that $\phi_1^{'}(t) \leq 0$, i.e.,
	\begin{equation}\label{m3}
	\dfrac{t}{f_0(\overline{F}_0^{-1}(t))} \geq \dfrac{\alpha}{\beta} \dfrac{t^{\dfrac{\alpha}{\beta}}}{g_0\left(\overline{G}_0^{-1}\left(t^{\dfrac{\alpha}{\beta}}\right)\right)}.
	\end{equation}
	Let $\overline{F}_0^{-1}(t) = z_1$ and $\overline{G}_0^{-1}\left(t^{\dfrac{\alpha}{\beta}}\right) = z_2$,
	\begin{align}\label{m3*}
	\dfrac{\overline{F}_0(z_1)}{f_0(z_1)} &\geq \dfrac{\alpha}{\beta}\dfrac{\overline{G}_0(z_2)}{g_0(z_2)}\nonumber\\
	\Rightarrow s_0(z_2) &\dfrac{\beta}{\alpha} \geq r_0(z_1),
	\end{align}
	where $r_0(z_1) = \dfrac{f_0(z_1)}{\overline{F}_0(z_1)}$ and $s_0(z_2) = \dfrac{g_0(z_2)}{\overline{G}_0(z_2)}$ . Under the hypothesis of the theorem 
	\begin{align*}
	\beta &\geq \alpha\\
	 \Rightarrow t=\overline{F}_0(z_1) &\leq t^{\left(\dfrac{\alpha}{\beta}\right)}=\overline{G}_0(z_2).
	\end{align*}
	 and $G_0 \leq_{hr} F_0 $  implies that $z_2 \leq z_1$ ($G_0 \leq_{st} F_0 $ follows from $G_0 \leq_{hr} F_0 $ subsequently we can derive that $\overline{G}_{0}(z_2)\geq \overline{F}_{0}(z_1) \geq \overline{G}_{0}(z_1)$. Finally the implication is possible as $\overline{G}_{0}$ is a decreasing function.) and $s_0(z_2) \geq r_0(z_2)$ . Also $F_0$ is DFR then, $z_2 \leq z_1 \Rightarrow r_0(z_2) \geq r_0(z_1)$.
	Combining all these we find that \eqref{m3*} holds true. Hence the result. \qed
	
\end{proof}
 The above result provides a general outlook over the PHR distributions. Apart from the fact that the component lifetimes are independent, it is observed that the above theorem is a generalization of Theorem 3.11 from \cite{var}. Here the baseline distributions are different, also the number of components are not same. The theorem holds true when we encounter a multiple-outlier model, we can observe that with the help of the following corollary.

\begin{cor}
	Let $X_1,X_2, \ldots, X_n$ be a set of n-independent random variables each belonging from a particular PHR multiple-outlier family with parameters \\$(\underbrace{\alpha_1, \ldots, \alpha_1}_{p_1}, \underbrace{\alpha_2, \ldots, \alpha_2}_{q_1})$ such that $p_1+q_1 = n_1$. We assume that $X_i \sim \overline{F}_i(x) = (\overline{F}_{0}(x))^{\alpha_1}$ for $i=1, 2,\ldots, p_1$ and  $X_i \sim \overline{F}_i(x) = (\overline{F}_{0}(x))^{\alpha_2}$ for $i= p_1+1,\ldots, n_1$. Also $Y_1, Y_2,\ldots, Y_n$ be another set of n-independent random variables each following PHR multiple-outlier family of distributions with a different distribution function and the parameter set is $(\underbrace{\beta_1, \ldots, \beta_1}_{p_2}, \underbrace{\beta_2, \ldots, \beta_2}_{q_2})$ and $p_2+q_2 = n_2$. Let $Y_i \sim \overline{G}_i(x) = (\overline{G}_{0}(x))^{\beta_1}$ for $i=1, 2,\ldots, p_2$ and  $Y_i \sim \overline{G}_i(x) = (\overline{G}_{0}(x))^{\beta_2}$ for $i= p_2+1,\ldots, n_2$. Under the assumption that $p_2\beta_1 + q_2\beta_2 \geq p_1\alpha_1 + q_1\alpha_2$, the baseline distribution function $F_{0}$ is DFR,  and $G_{0} \leq_{hr} F_{0}$ then $Y_{1:n_2} \leq_{disp} X_{1:n_1}$.
\end{cor}
Since $F_0 \leq_{hr} F_0$, hence the above corollary holds when the baseline distributions are same. Some examples are given here that satisfies the condition given in the above theorem. Since variance is a measure of dispersion, $X <_{disp} Y \Rightarrow V(X) < V(Y)$.
\begin{example}
	\normalfont Let ($X_1, X_2, X_3$) and ($Y_1, Y_2, Y_3$) be independent Weibull distributed random variables. The survival function of $X_i$ is $\overline{F}_{k_i}(t)= \exp(-k_it^{\alpha})$ and corresponding to $Y_i$ is $\overline{G}_{k_i^{*}}(t)= \exp(-k_i^{*}t^{\alpha})$, such that the baseline survival function is $\overline{F}(t)= \exp(-t^{\alpha})$. Note that the baseline distribution is DFR when the shape parameter $\alpha < 1$. Consider $\alpha = 0.7, $ $k_1=1.7, k_2 = 2, k_3 = 0.9$ and $k_1^{*} = 1, k_2^{*} = 3, k_3^{*} = 2.3$, here $\displaystyle\sum_{i=1}^3k_i = 4.6$ and $\displaystyle\sum_{i=1}^3k_i^{*} = 6.3$. Then using the formula for variance as
	$$V(X_{1:3})= \left(\dfrac{1}{k_1+k_2+k_3}\right)^{2/\alpha}\left(\Gamma\left(\dfrac{2}{\alpha}+1\right)-\left(\Gamma\left(\dfrac{1}{\alpha}+1\right)\right)^2\right).$$
	We hence obtain $V(X_{1:3})= 0.043782$ and $V(Y_{1:3})= 0.017826$, $V(X_{1:3}) > V(Y_{1:3})$. Thus the results obtained earlier holds in this case.
\end{example}

\begin{example}
	\normalfont Let ($X_1, X_2, X_3$) and ($Y_1, Y_2, Y_3$) be independent Pareto type II (Lomax distributed) random variable with scale parameter =1 (say), such that each $X_i$ has survival function $\overline{F}_{\alpha_i}(t) = (1+t)^{-\alpha_i}$ and $Y_i$ has survival function $\overline{G}_{\alpha_i^{*}}(t) = (1+t)^{-\alpha_i^{*}}$. Then the expectation and variance of $X_{1:3}$ are given by 
	$$E(X_{1:3}) = \dfrac{1}{(\alpha_1 +\alpha_2 +\alpha_3 -1)}, \alpha_1 +\alpha_2 +\alpha_3 > 1; $$
	$$V(X_{1:3}) = \dfrac{\alpha_1 +\alpha_2 +\alpha_3}{(\alpha_1 +\alpha_2 +\alpha_3-2)(\alpha_1 +\alpha_2 +\alpha_3-1)^2}, \alpha_1 +\alpha_2 +\alpha_3 > 2.$$
	If the values for $(\alpha_1 ,\alpha_2 ,\alpha_3)$ and $(\alpha^{*}_1 ,\alpha^{*}_2 ,\alpha^{*}_3)$ be considered as $(1,4,7)$ and $(1.2,3.5,7.2)$ respectively. The variances are given by \\
	$V(X_{1:3}) = 0.009917$, $V(Y_{1:3}) = 0.010117$, i.e., $V(X_{1:3}) < V(Y_{1:3})$. The baseline distribution $\overline{F}_1(t) = (1+t)^{-1}$ is DFR and the sum of the shape parameters, $\sum_{i=1}^3 \alpha_i = 12$ and $\sum_{i=1}^3 \alpha_i^{*} = 11.9$. This supports the result obtained in the previous theorem.
\end{example}
Here we discuss the following result for a parallel system with components following multiple outlier PRHR model.

\begin{theorem}\label{newthm2}
Let $X_1,X_2,\ldots,X_{n_1}$ be a n-independent set of random variables each belonging from PRHR family of distributions with parameters $(\alpha_1,\alpha_2\ldots, \alpha_{n_1})$, such that 
$X_i \sim F_i(x) = (F_{0}(x))^{\alpha_i}$ for $i=1, 2,\ldots, n_1$. Also $Y_1, Y_2,\ldots, Y_{n_2}$ be another set of n-independent random variables each following PRHR family of distributions with a different distribution function and the parameter set is $(\beta_1, \beta_2,\ldots, \beta_{n_2})$. Let $Y_i \sim G_i(x) = (G_{0}(x))^{\beta_i}$ for $i=1, 2,\ldots, n_2$. Then $Y_{n_2:n_2} \leq_{disp} X_{n_1:n_1}$ if the baseline distribution $F_0$ follows IRHR model, $\displaystyle\sum_{i=1}^{n_2}\beta_i \geq \displaystyle\sum_{i=1}^{n_1}\alpha_i$ and $F_0 \leq_{rh} G_0$.
\end{theorem}
\begin{proof}
The distribution function of $X_{n_1:n_1}$ and $Y_{n_2:n_2}$ are
\begin{align*}
F_{n_1:n_1}(x) &= [F_0(x)]^{\displaystyle\sum_{i=1}^{n_1} \alpha_i},\  \text{and},\\
G_{n_2:n_2}(x) &= [G_0(x)]^{\displaystyle\sum_{i=1}^{n_2} \beta_i}
\end{align*}
respectively. Similar to the previous theorem, we take $\displaystyle\sum_{i=1}^{n_1} \alpha_i = \alpha$ and $\displaystyle\sum_{i=1}^{n_2} \beta_i = \beta$. Let 
\begin{align*}
\psi_2(y) &= F_{n_1:n_1}^{-1}(y) - G_{n_2:n_2}^{-1}(y)\\
&= F_0^{-1}(y^{1/\alpha}) - G_0^{-1}(y^{1/\beta}).
\end{align*}
We are required to prove that $Y_{n_2:n_2} \leq_{disp} X_{n_1:n_1}$, i.e., $\psi_2(y)$ is increasing in $y \in (0,1)$.\\  Hence $Y_{n_2:n_2} \leq_{disp} X_{n_1:n_1}$ if and only if   $\phi_2(t) = F_0^{-1}(t) - G_0^{-1}\left(t^{\dfrac{\alpha}{\beta}}\right)$ is increasing in $t \in (0,1)$, where $t=y^{1/\alpha}$. Note that
\begin{align*}
\phi_2^{'}(t) = \dfrac{1}{f_0(F_0^{-1}(t))} - \dfrac{\alpha}{\beta} \dfrac{t^{\dfrac{\alpha}{\beta} - 1}}{g_0\left(G_0^{-1}\left(t^{\dfrac{\alpha}{\beta}}\right)\right)}.
\end{align*}
We need to show that $\phi_2^{'}(t) \geq 0$, i.e., 
\begin{equation}\label{m2}
\dfrac{t}{f_0(F_0^{-1}(t))} \geq \dfrac{\alpha}{\beta} \dfrac{t^{\dfrac{\alpha}{\beta}}}{g_0\left(G_0^{-1}\left(t^{\dfrac{\alpha}{\beta}}\right)\right)}.
\end{equation}
Put $F_0^{-1}(t) = z_1$ and $G_0^{-1}\left(t^{\dfrac{\alpha}{\beta}}\right) = z_2$. From \eqref{m2} it is sufficient to show
\begin{align}\label{m21}
\dfrac{F_0(z_1)}{f_0(z_1)} &\geq \dfrac{\alpha}{\beta} \dfrac{G_0(z_2)}{g_0(z_2)} \nonumber \\
\Leftrightarrow \tilde{s}_0(z_2)&\dfrac{\beta}{\alpha} \geq \tilde{r}_0(z_1).
\end{align}
As  
\begin{align*}
\dfrac{\beta}{\alpha} & \geq 1\\
\Rightarrow t= F_0(z_1)& \leq t^{\dfrac{\alpha}{\beta}}= G_o(z_2).
\end{align*}
Since $F_0 \leq _{rh} G_0$ implies $F_0 \leq_{st} G_0$, hence $G_0(z_1)\leq F_0(z_1)\leq G_0(z_2)$ i.e., $z_1 \leq z_2$. Again $F_0$ follows increasing reversed hazard rate (IRHR) model hence $z_1 \leq z_2 \Rightarrow \tilde{r}_0(z_1)\leq \tilde{r}_0(z_2)$.
  Lastly, $F_0 \leq_{rh} G_0 \Rightarrow \tilde{r}_0(x)\leq \tilde{s}_0(x)$ for all $x$, thus $\tilde{r}_0(z_2)\leq \tilde{s}_0(z_2)$. Combining these inequalities we obtain the required result. \qed
\end{proof}

Similar to Theorem 1 we can again obtain a corollary for multiple-outlier model from PRHR distributions.

\begin{cor}
Let $X_1,X_2,\ldots,X_{n_1}$ be a n-independent set of multiple-outlier random variables each belonging from PRHR family of distributions with parameters \\$(\underbrace{\alpha_1,\ldots, \alpha_1}_{p_1},\underbrace{\alpha_2,\ldots,\alpha_2}_{q_1})$, $p_1+q_1=n_1$ We assume that $X_i \sim F_i(x) = (F_{0}(x))^{\alpha_1}$ for $i=1, 2,\ldots, p_1$ and  $X_i \sim F_i(x) = (F_{0}(x))^{\alpha_2}$ for $i= p_1+1,\ldots, n_1$. Also $Y_1, Y_2,\ldots, Y_{n_2}$ be another set of n-independent random variables each following proportional reversed hazard rate multiple-outlier family of distributions with a different distribution function and the parameter set is $(\underbrace{\beta_1, \ldots, \beta_1}_{p_2}, \underbrace{\beta_2, \ldots, \beta_2}_{q_2})$ and $p_2+q_2 = n_2$. Let $Y_i \sim G_i(x) = (G_{0}(x))^{\beta_1}$ for $i=1, 2,\ldots, p_2$ and  $Y_i \sim G_i(x) = (G_{0}(x))^{\beta_2}$ for $i= p_2+1,\ldots, n_2$. Then $Y_{n_2:n_2} \leq_{disp} X_{n_1:n_1}$ if the baseline distribution follows IRHR model, $F_0 \leq_{rh} G_0$ and $p_2\beta_1 + q_2\beta_2 \geq p_1\alpha_1 + q_1\alpha_2$.
\end{cor}
In the next theorem, we provide a result for series systems with unequal number of components following PHR models with different baseline distributions.


\begin{theorem}
Consider a system of $n_1$ components, where the lifetime of each component is represented by the random variable $X_1,X_2,\ldots, X_{n_1}$ respectively such that each of $X_1,\ldots,X_{p_1}$ has survival function $[\overline{F}(x)]^{\alpha_i}$, $i=1,2,\ldots,p_1$ and \\$X_{p_1+1},\ldots,X_{n_1}$ has survival function $[\overline{G}
(x)]^{\alpha_i}$, $i=p_1+1,p_1+2,\ldots,n_1$. Similarly another system with $n_2$ components is considered where the components $Y_1,\ldots,Y_{p_2}$ has survival function as $[\overline{F}(x)]^{\beta_i}$, $i=1,2,\ldots,p_2$ whereas the components $Y_{p_2+1},\ldots,Y_n$ has survival function $[\overline{G}
(x)]^{\beta_i}$,  $i=p_2+1,p_2+2,\ldots,n_2$. Then $X_{1:n_1} \leq_{hr} Y_{1:n_2}$ whenever $\displaystyle\sum_{i=1}^{p_1}\alpha_i > \displaystyle\sum_{i=1}^{p_2}\beta_i$ and $\displaystyle\sum_{i=p_1+1}^{n_1}\alpha_i  > \displaystyle\sum_{i=p_2+1}^{n_2}\beta_i$.
\end{theorem}
 
\begin{proof}
The survival function of $X_{1:n_1}$ is 
$$\overline{F}_{1:n_1}(x)=[\overline{F}(x)]^{\displaystyle\sum_{i=1}^{p_1}\alpha_i}[\overline{G}(x)]^{\displaystyle\sum_{i=p_1+1}^{n_1}\alpha_i},$$
and the survival function of $Y_{1:n_2}$ is 
$$\overline{G}_{1:n_2}(x)=[\overline{F}(x)]^{\displaystyle\sum_{i=1}^{p_2}\beta_i}[\overline{G}(x)]^{\displaystyle\sum_{i=p_2+1}^{n_2}\beta_i}.$$
Consider the ratio 
\begin{equation}\label{ratio}
\dfrac{\overline{F}_{1:n_1}(x)}{\overline{G}_{1:n_2}(x)}= [\overline{F}(x)]^{\left(\displaystyle\sum_{i=1}^{p_1}\alpha_i - \displaystyle\sum_{i=1}^{p_2}\beta_i\right)}[\overline{G}(x)]^{\left(\displaystyle\sum_{i=p_1+1}^{n_1}\alpha_i - \displaystyle\sum_{i=p_2+1}^{n_2}\beta_i\right)}
\end{equation}
Differentiating \eqref{ratio} with respect to x,
\begin{align*}
\dfrac{d}{dx}\left(\dfrac{\overline{F}_{1:n}(x)}{\overline{G}_{1:n}(x)}\right)&= - \dfrac{\overline{F}_{1:n_1}(x)}{\overline{G}_{1:n_2}(x)}\left(\left(\displaystyle\sum_{i=1}^{p_1}\alpha_i - \displaystyle\sum_{i=1}^{p_2}\beta_i\right)\dfrac{f(x)}{\overline{F}(x)} +  \left(\displaystyle\sum_{i=p_1+1}^{n_1}\alpha_i  - \displaystyle\sum_{i=p_2+1}^{n_2}\beta_i\right)\dfrac{g(x)}{\overline{G}(x)}\right)\\
&<0,
\end{align*}
whenever $\displaystyle\sum_{i=1}^{p_1}\alpha_i > \displaystyle\sum_{i=1}^{p_2}\beta_i$ and $\displaystyle\sum_{i=p_1+1}^{n_1}\alpha_i  > \displaystyle\sum_{i=p_2+1}^{n_2}\beta_i$. Hence the result follows. \qed
\end{proof}

If we consider a similar problem wherein the random variables follows a PRHR distribution, we arrive at the following theorem.

\begin{theorem}
	Consider an independent set of $n_1$ random variables \\$X_1,X_2,\ldots, X_{p_1}, X_{p_1+1},\ldots, X_{n_1}$ such that the distribution function of $X_i$, $F_{X_i}(x)=[F(x)]^{\alpha_i}$ for $i=1,2,\ldots,p_1$ and $F_{X_i}(x)=[G(x)]^{\alpha_i}$ for $i=p_1+1,\ldots,n_1$. Another set of $n_2$ independent components $Y_1,Y_2,\ldots, Y_{p_2}, Y_{p_2+1},\ldots, Y_{n_2}$ are such that the distribution function of $Y_i$, $F_{Y_i}(x) = [F(x)]^{\beta_i}$, $i=1,2,\ldots,p_2$ and $F_{Y_i}(x) = [G(x)]^{\beta_i}$ for $i=p_2+1,\ldots,n_2$. Then $X_{n_1:n_1} \geq_{rh} Y_{n_2:n_2}$ whenever $\displaystyle\sum_{i=1}^{p_1}\alpha_i > \displaystyle\sum_{i=1}^{p_2}\beta_i$ and $\displaystyle\sum_{i=p_1+1}^{n_1}\alpha_i  > \displaystyle\sum_{i=p_2+1}^{n_2}\beta_i$.
\end{theorem}

\begin{proof}
	The distribution functions of $X_{n_1:n_1}$ and $Y_{n_2:n_2}$ are
	$$F_{n_1:n_1}(x)= [F(x)]^{\displaystyle\sum_{i=1}^{p_1}\alpha_i}[G(x)]^{\displaystyle\sum_{i=p_1+1}^{n_1}\alpha_i},$$
	$$G_{n_2:n_2}(x)= [F(x)]^{\displaystyle\sum_{i=1}^{p_2}\beta_i}[G(x)]^{\displaystyle\sum_{i=p_2+1}^{n_2}\beta_i}$$
	respectively.
	Differentiating the ratio $\dfrac{F_{n_1:n_1}(x)}{G_{n_2:n_2}(x)}$ with respect to $x$, we observe,
	\begin{align*}
	\dfrac{d}{dx}\left(\dfrac{F_{n_1:n_1}(x)}{G_{n_2:n_2}(x)}\right) &= \dfrac{F_{n_1:n_1}(x)}{G_{n_2:n_2}(x)} \left(\left(\displaystyle\sum_{i=1}^{p_1}\alpha_i - \displaystyle\sum_{i=1}^{p_2}\beta_i\right)\dfrac{f(x)}{F(x)}+ \left( \displaystyle\sum_{i=p_1+1}^{n_1}\alpha_i  - \displaystyle\sum_{i=p_2+1}^{n_2}\beta_i\right)\dfrac{g(x)}{G(x)}\right)\\
	& >0,
	\end{align*}
	whenever $\displaystyle\sum_{i=1}^{p_1}\alpha_i > \displaystyle\sum_{i=1}^{p_2}\beta_i$ and $\displaystyle\sum_{i=p_1+1}^{n_1}\alpha_i  > \displaystyle\sum_{i=p_2+1}^{n_2}\beta_i$ and the result follows. \qed
\end{proof}
The above theorem deals with a more complex set of parameters and baseline distributions as compared to that of theorem 3.7 from \cite{var} where the component lifetimes are dependent but the parameters are restricted and the baseline distributions are all same. 
\section{Star ordering results for unequal sample sizes}
In this section we present a comparison between two systems based on star ordering. The random variables are same as used in the previous section. 
Consider a series system with components following PHR model and have unequal sample sizes.

\begin{theorem}
Let $X_1,X_2,\ldots,X_{n_1}$ be a $n_1$-independent set of non-negative random variables such that $X_i \sim [\overline{F}(x)]^{\alpha_i}$ for $i=1,2,\ldots,n_1$ and $Y_1,Y_2,\ldots,Y_{n_2}$ be another $n_2$-independent set of non-negative random variables such that $Y_i \sim [\overline{F}(x)]^{\beta_i}$ for $i=1,2,\ldots,n_2$, where $n_1$ and $n_2$ may or may not be the same. Then
\item $\displaystyle\sum_{i=1}^{n_1}\alpha_i \leq \displaystyle\sum_{i=1}^{n_2}\beta_i \Rightarrow X_{1:n_1} \geq_{*} Y_{1:n_2}$, whenever $x r(x)$ is decreasing.
\end{theorem}
\begin{proof} 
The survival function of $X_{1:n_1}$ is 
\begin{equation}
\overline{F}_{1:n_1}(x)=[\overline{F}(x)]^{\displaystyle\sum_{i=1}^{n_1}\alpha_i}.
\end{equation}
Let $\displaystyle\sum_{i=1}^{n_1}\alpha_i= \alpha$, then $\overline{F}_{1:n_1}(x)=[\overline{F}(x)]^{\alpha}= \overline{F}_{\alpha}(x)$ (say).\\ The corresponding probability density function is 
\begin{align*}
f_{X_{1:n_1}}(x)&=\alpha f(x) [\overline{F}(x)]^{\alpha-1}\\
&=f_{\alpha}(x).
\end{align*}
Note that the ratio
\begin{equation}
\dfrac{F_{\alpha}^{\prime}(x)}{f_{\alpha}(x)}=-\dfrac{1}{\alpha}\dfrac{\ln\overline{F}(x)}{r(x)},
\end{equation}
where $F_{\alpha}(x)=1-[\overline{F}(x)]^{\alpha}$ and $F_{\alpha}^{\prime}(x)=\frac{d}{d \alpha}F_{\alpha}(x)$.\\
The  theorem follows by differentiating the ratio $\dfrac{F_{\alpha}^{\prime}(x)}{x f_{\alpha}(x)}$ with respect to $x$.\\
Note that 
\begin{align*}
\dfrac{d}{dx}\left(\dfrac{F_{\alpha}^{\prime}(x)}{x f_{\alpha}(x)}\right)&=\dfrac{1}{\alpha}\left(\dfrac{x(r(x))^2+(xr^{\prime}(x)+r(x))\ln\overline{F}(x)}{(xr(x))^2}\right)\\
&>0,
\end{align*}
whenever $xr(x)$ is decreasing in $x$. Now using Lemma \ref{lemma4}, we obtain $X_{1:n_1} \geq_{*} Y_{1:n_2}$ whenever $\displaystyle\sum_{i=1}^{n_1}\alpha_i \leq \displaystyle\sum_{i=1}^{n_2}\beta_i$. \qed
\end{proof}

We observe here that the hazard rate functions of $X_{1:n_1}$ and $Y_{1:n_2}$  are $r_{X_{1:n_1}}(x)=\displaystyle\sum_{i=1}^{n_1}\alpha_i r(x)$ and $r_{Y_{1:n_2}}(x)=\displaystyle\sum_{i=1}^{n_2}\beta_i r(x)$ respectively, where $r(x)$ is the hazard rate function of baseline distribution $F(x)$. Then 
$$r_{X_{1:n_1}}(x) \leq r_{Y_{1:n_2}}(x) \text{ whenever } \displaystyle\sum_{i=1}^{n_1}\alpha_i \leq \displaystyle\sum_{i=1}^{n_2}\beta_i.$$
The above result is applicable for multiple-outlier models where the parameters are as described in Corollary 1. Moreover this theorem can be considered as a more general form of theorem 3.9 from \cite{var}. Here the parameters are all different and only a simple inequality exists between them.\\

Next we consider a reversed hazard rate ordering result for the parallel system having Pareto distributed components such that the sample sizes are equal. 

\begin{theorem}
Let $X_1,X_2,\ldots,X_n$ and $Y_1,Y_2,\ldots,Y_n$ be two sets of n-independent Pareto distributed random variables such that the survival function of $X_i$ is $\overline{F}_i(x)=\left(1+\dfrac{x}{\theta}\right)^{-\alpha_i} ,~x>0,~\theta>0,~\alpha_i>0$ and that of $Y_i$ is $\overline{G}_i(x)=\left(1+\dfrac{x}{\theta}\right)^{-\alpha_i^*} ,~x>0,~\theta>0,~\alpha_i^*>0$. Let $\underline{\alpha}=(\alpha_1,\alpha_2,\ldots,\alpha_n)$, $\underline{\alpha}^*=(\alpha_1^*,\alpha_2^*,\ldots,\alpha_n^*)$, then $\underline{\alpha} \prec^{w} \underline{\alpha}^*$ $\Rightarrow ~ X_{n:n}\leq_{rh} Y_{n:n}$.
\end{theorem}

\begin{proof}
The distribution function of $X_{n:n}$ is
\begin{equation}
F_{X_{n:n}}(x)=\displaystyle\prod_{i=1}^n\left[1-\left(1+\dfrac{x}{\theta}\right)^{-\alpha_i}\right],
\end{equation}
and the corresponding reversed hazard rate function is
\begin{equation}
\tilde{r}_{X_{n:n}}(x)=\dfrac{1}{x+\theta}\displaystyle\sum_{i=1}^n g(\alpha_i),
\end{equation}
where $g(\alpha)=\dfrac{\alpha}{\left(\dfrac{x}{\theta}+1\right)^{\alpha}-1}$. Let $u=\left(\dfrac{x}{\theta}+1\right)^{\alpha}$ and $u>1$ such that $g(\alpha)=\dfrac{\alpha}{u^{\alpha}-1}$. Now,
\begin{equation}
g'(\alpha)=\dfrac{u^{\alpha}(1-\alpha\ln u)-1}{(u^{\alpha}-1)^2}
\end{equation}
and
\begin{equation}
g''(\alpha)=\dfrac{u^{\alpha}\ln u((u^{\alpha}\ln u+ \ln u)\alpha-2u^{\alpha}+2)}{(u^{\alpha}-1)^3}.
\end{equation}
$g''(\alpha)\overset{sgn}{=}u^{\alpha}(\ln u)  \phi(u)$, where $\phi(u)= (u^{\alpha}\ln u+ \ln u)\alpha-2u^{\alpha}+2$, $\phi(1)=0$
Also,
\begin{align*}
\phi'(u)&={\alpha}^2u^{\alpha-1}\ln u +\dfrac{\alpha}{u}-\alpha u^{\alpha-1}\\
&=\dfrac{\alpha}{u} \phi_1(u),
\end{align*}
such that $\phi_1(u)=\alpha u^{\alpha}\ln u +1-u^{\alpha}$ and $\phi_1(1)=0$. And 
\begin{align*}
\phi_1'(u)&={\alpha}^2u^{\alpha-1} \ln u\\
&>0.
\end{align*}
Hence it is observed that $g''(\alpha) > 0$ for $x ~>~0~(u~>~1)$, i.e., $g(\alpha)$ is convex in $\alpha$. Hence, using Lemma \ref{lemma2} we obtain, $\tilde{r}_{X_{n:n}}(x)$ is Schur convex w.r.t $\underline{\alpha}$. Moreover, 
\begin{align*}
g'(\alpha)&\overset{sgn}{=} u^{\alpha}(1-\alpha\ln u)-1, ~u>1\\
&=h(u) \text{ say, }
\end{align*} 
then $h'(u)= -{\alpha}^2u^{\alpha-1}\ln u$. Also $h(1)=0$, then $g'(\alpha) < 0$ for $x>0~(u>1)$. Thus $\tilde{r}_{X_{n:n}}(x)$ is decreasing in $\underline{\alpha}$ and Schur convex w.r.t $\underline{\alpha}$. Using Lemma \ref{lemma3}, we infer that $\underline{\alpha} \prec^{w} \underline{\alpha}^*$ $\Rightarrow ~ \tilde{r}_{X_{n:n}}(x) \leq \tilde{r}_{Y_{n:n}}(x)$. Hence the result follows. \qed
\end{proof}
\section{Dependent model}
In this section we have considered a dependent set of random variables instead of independent random variables as discussed in the earlier sections. Hence we shall observe few definitions required especially to study the dependent models.
\begin{defn}\textbf{Survival copula:}
Let $(X_1, \ldots, X_n)$ be a n-dimensional random vector defined on a probability space $(\Omega, \mathbb{F}, \mathbb{P})$, the multivariate survival function is defined as $$\overline{F}(x_1, \ldots, x_n) = P[X_1 > x_1, \ldots, X_n > x_n] = \tilde{C}(\overline{F}_1(x_1),\ldots, \overline{F}_n(x_n)), ~x_1, \ldots, x_n \in \mathbb{R},$$ where  $\tilde{C}$ is the n-dimensional survival copula of the random vector $(X_1, \ldots, X_n)$. \\
$\tilde{C}$ is a continuous function defined over the n-dimensional space as $\tilde{C} : [0,1]^n \mapsto [0,1]$, to develop multivariate survival functions from the marginal survival functions.
\end{defn}
Archimedean copula is a very widely used class of survival copula because of its analytical tractability.
\begin{defn}\textbf{Archimedean copula}
	
A n-dimensional \emph{Archimedean copula} $\tilde{C} : [0,1]^n \mapsto [0,1]$  is represented as  $$\tilde{C}(u_1, \dots, u_n) = \psi(\psi^{-1}(u_1)+\ldots+\psi^{-1}(u_n))  , ~u_k \in [0,1] ~\text{for}~ k=1, \dots, n,$$

where the survival copula $\tilde{C}$ is generated by the generator function (also known as Archimedean generator function) $\psi : [0,\infty) \mapsto [0,1]$, $\psi$ is  n -monotone ($n \geq 2$) over an open interval $I$ $\subset \mathbb{R}$ (where the end points of the interval $I$ belongs to the limit point of $\mathbb{R}$) if $\psi$ has derivatives upto order $n-2$ and 
$$(-1)^{r}{\psi}^{(r)}(x) \geq 0 ~~ \text{for } r=0,1,2,\ldots,n-2$$
for any $x \in I$ and also $(-1)^{(n-2)}{\psi}^{(n-2)}$ is non-increasing and convex over $I$.  $\phi=\psi^{-1}$ is the corresponding inverse function. Clayton copula, Frank copula are few archimedean copulas studied in the literature. 
\end{defn}

For a detailed discussion on Archimedean Copula one can refer to \cite{18}.
Now we shall discuss a lemma, that is used further in understanding the forth coming theorems. 

\begin{lemma} \label{lemmaj1}
Let $Y_1,Y_2,\ldots, Y_n$ be n random variables such that $Y_i=X-\mu_i$, where $\mu_i$ for $i=1,2,\ldots,n$ ($P[X>x]=\overline{F}(x)$) are the corresponding location parameters respectively, then the survival function of the minimum of $Y_1,Y_2,\ldots, Y_n$ \\($P[\min\{Y_1,Y_2,\ldots, Y_n\}> x]$) is given by
\normalfont 
\begin{equation}\label{J1}
 J_1(\underline{\mu};\overline{F}(x),\psi) = \psi(\sum_{k=1}^n \phi(\overline{F}(x+\mu_k))) , 
 \end{equation}  
\begin{enumerate}
	\item $J_1$is decreasing in $\mu_i$ for each $i$,
	\item $J_1$ is Schur-concave (Schur-convex) in $\underline{\mu}$ whenever $\psi$ is log-convex (log-concave) and $F$ is IFR (DFR) distribution.
\end{enumerate}    	
	\end{lemma}
\begin{proof}
\normalfont
Consider 
\begin{equation} 
J_1(\underline{\mu};\overline{F}(x),\psi) = \psi(\sum_{k=1}^n \phi(\overline{F}(x+\mu_k))).  \label{j1} 
\end{equation}\\ Differentiating $\eqref{j1}$ w.r.t $\mu_i$, $i=1,\ldots,n$
\begin{equation}
\frac{\partial J_1}{\partial \mu_i} = - \psi^{\prime}\left(\sum_{k=1}^n \phi (\overline{F}(x+ \mu_k))\right)\frac{f(x+\mu_i)}{\overline{F}(x+\mu_i)}\cdot \frac{\psi(\phi(\overline{F}(x+\mu_i)))}{\psi^{\prime}(\phi(\overline{F}(x+\mu_i)))} \leq 0 ~\forall x \text{ and } \forall i.
\end{equation}
$\therefore J_1$ is decreasing in $\mu_i$.\\
Let $r(x) = \dfrac{f(x)}{\overline{F}(x)}$ and $\rho_1(\mu_i;x) = \dfrac{f(x +\mu_i)}{\overline{F}(x+ \mu_i)} \cdot\dfrac{\psi(\phi(\overline{F}(x+\mu_i)))}{\psi^{\prime}(\phi(\overline{F}(x+\mu_i)))}. $\\

$r(x+\mu_i)$ is increasing (decreasing) in $\mu_i$ and $\psi$ is log-convex (log-concave) $\Rightarrow \rho_1(\mu_i;x)$ is decreasing (increasing) in $\mu_i$.\\
Then for $i \neq j$,
\begin{align*}
\Delta &= (\mu_i -\mu_j)\left(\frac{\partial J_1}{\partial \mu_i} - \frac{\partial J_1}{\partial \mu_j}\right)\\
&= - \psi^{\prime}\left(\sum_{k=1}^n \phi (\overline{F}(x+ \mu_k))\right)\cdot (\mu_i -\mu_j)\cdot(\rho_1(\mu_i;x) - \rho_1(\mu_j;x))\\
& \leq (\geq) 0.
\end{align*}
Using Lemma \ref{lemma1} $J_1$ is Schur-concave (Schur-convex) in $\underline{\mu}$. \qed
\end{proof}
Earlier \cite{min} and \cite{max} have discussed about the stochastic ordering between two systems where the component lifetimes are independent and each belongs from a location-scale family, necessarily with the same baseline distribution function. The next theorems establishes the conditions under which a series system (parallel) can be compared with another series (parallel) system, where all the  component lifetimes are dependent and the baseline distribution functions are different.
\begin{theorem}{\label{ac1}}
Let $Y_1,Y_2,\ldots, Y_n$ be n random variables such that $Y_i=X-\mu_i$,  ($P[X>x]=\overline{F}(x)$) where $\mu_i$ for $i=1,2,\ldots,n$ are the corresponding location parameters respectively, then the survival function of the minimum of $Y_1,Y_2,\ldots, Y_n$ ($P[\min\{Y_1,Y_2,\ldots, Y_n\}> x]$) is given by
\normalfont \[ J_1(\underline{\mu};\overline{F}(x),\psi_1) = \psi_1(\sum_{k=1}^n \phi_1(\overline{F}(x+\mu_k))) , \]  $\psi_1$ is log-convex (log-concave) and $F$ is IFR (DFR) distribution.
If there exists another set of n random variables $Z_1,Z_2,\ldots,Z_n$ ($Z_i=W-\mu_i^{*}$ and $P[W>x]=\overline{G}(x)$) such that the survival function for the minimum of $Z_1,Z_2,\ldots,Z_n$ is 
\normalfont \[ J_1(\underline{\mu}^{*};\overline{G}(x),\psi_2) = \psi_2(\sum_{k=1}^n \phi_2(\overline{G}(x+\mu_k^{*}))) , \] then as $(\mu_1,\mu_2,\ldots,\mu_n) \prec_{w} (\prec^{w}) ~(\mu_1^{*},\mu_2^{*},\ldots,\mu_n^{*})$ we obtain $Y_{1:n}\geq_{st} (\leq_{st}) Z_{1:n}$ as $\psi$ is log-convex (log-concave) , $X \geq_{st} W$ and $F$ is IFR (DFR) distribution.
\end{theorem}
\begin{proof} 
 It is observed that $(\mu_1,\mu_2,\ldots,\mu_n) \prec_{w} (\prec^{w})~ (\mu_1^{*},\mu_2^{*},\ldots,\mu_n^{*}) \Rightarrow J_1(\underline{\mu};\overline{F}(x),\psi_1) \geq (\leq)~ J_1(\underline{\mu}^{*};\overline{F}(x),\psi_1)$ using Lemma \ref{lemmaj1} and Lemma \ref{lemma3}. Now we are required to show that $\psi_1(\displaystyle\sum_{k=1}^n \phi_1(\overline{F}(x+\mu_k))) \geq (\leq) \psi_2(\displaystyle\sum_{k=1}^n \phi_2(\overline{G}(x+\mu_k^{*})))$. 
 $$\psi_1(\displaystyle\sum_{k=1}^n \phi_1(\overline{F}(x+\mu_k^{*}))) \geq (\leq)\psi_2(\displaystyle\sum_{k=1}^n \phi_2(\overline{F}(x+\mu_k^{*})))$$
 as $\phi_1\cdot\psi_2 (\phi_2\cdot\psi_1)$ is super-additive. Also using the condition $X \geq_{st} (\leq_{st}) W$ we have $\overline{F}(x+\mu_k^{*}) > (<) \overline{G}(x+\mu_k^{*})$ for all $x$, this further implies that $\psi_2(\displaystyle\sum_{k=1}^n\phi_2(\overline{F}(x+\mu_k^{*}))) > (<) \psi_2(\displaystyle\sum_{k=1}^n\phi_2(\overline{G}(x+\mu_k^{*})))$.\\
 Hence the result follows. \qed
\end{proof}

Next we introduce $J_2(\underline{\mu};F(x),\psi) =1-\psi\left(\displaystyle\sum_{k=1}^n \phi(F(x + \mu_k)) \right)$. It can be observed that $J_2(\underline{\mu};x,\psi)$ is the survival function of the maximum order statistic from the set $Y_1,Y_2,\ldots,Y_n$ where $Y_i= X- \mu_i$ for each $i=1,2,\ldots,n$ , the components are dependent (we deal here with the same set of components as used in theorem \ref{ac1}).
\begin{lemma}\label{lemmaj2}
	$$J_2(\underline{\mu};F(x),\psi) =1-\psi\left(\displaystyle\sum_{k=1}^n \phi(F(x + \mu_k)) \right)$$
	$J_2(\underline{\mu};F(x),\psi)$	is decreasing in $\mu_i$ for each $i$, and $J_2$ is Schur-concave (Schur-convex) in $\underline{\mu}$ whenever $\psi$ is log-convex (log-concave) and $F$ is IRFR(DRFR) distribution.
\end{lemma}
\begin{proof}
\begin{equation} \label{j2}
J_2(\underline{\mu};x,\psi) = 1-\psi\left(\displaystyle\sum_{k=1}^n \phi(F(x + \mu_k)) \right)
\end{equation}
Differentiating \eqref{j2} with respect to each $\mu_i$, we have 
\begin{equation} 
\frac{\partial J_2}{\partial \mu_i} = -\psi^{\prime}\left(\displaystyle\sum_{k=1}^n \phi(F(x + \mu_k)) \right) \cdot \dfrac{f(x+\mu_i)}{F(x+\mu_i)} \cdot \dfrac{\psi(\phi(F(x+\mu_i)))}{\psi^{\prime}(\phi(F(x+\mu_i)))} ~ \leq 0.
\end{equation}
$\therefore ~ J_2$ is decreasing in $\mu_i$.\\

Let $\tilde{r}(x) =\dfrac{f(x)}{F(x)}$ and $\rho_2(\mu_i;x) = \dfrac{f(x+\mu_i)}{F(x+\mu_i)} \cdot \dfrac{\psi(\phi(F(x+\mu_i)))}{\psi^{\prime}(\phi(F(x+\mu_i)))}$.\\

$\tilde{r}(x +\mu_i)$ is increasing (decreasing) in $\mu_i$ and $\psi$ is log-convex (log-concave) $\Rightarrow \rho_2(\mu_i;x)$ is decreasing (increasing) in $\mu_i$.\\
$\therefore$ for $i \neq j$,
\begin{align*}
\Delta &= (\mu_i -\mu_j)\left(\frac{\partial J_2}{\partial \mu_i} - \frac{\partial J_2}{\partial \mu_j}\right)\\ 
&=  - \psi^{\prime}\left(\displaystyle\sum_{k=1}^n \phi({F}(x+ \mu_k))\right)\cdot(\mu_i -\mu_j)(\rho_2(\mu_i;x) - \rho_2(\mu_j;x)) \\
&= \leq (\geq) 0.
\end{align*}
Thus using Lemma \ref{lemma1}, $J_2$ is Schur-concave (Schur-convex) in $\underline{\mu}$. \qed
\end{proof}
\begin{theorem}
Let $Y_1,\ldots, Y_n$ and $Z_1,\ldots,Z_n$ be two n-dimensional random variables such that $Y_i= X-\mu_i$ and $Z_i = W-\mu_i*$, $i=1,2,\ldots,n$. Then
$$\underline{\mu} \prec_{w} \underline{\mu}^{*}~ (\underline{\mu} \prec^{w} \underline{\mu^{*}}) ~ \Rightarrow Y_{n:n} \geq_{st}(\leq_{st}) Z_{n:n}$$
whenever  $\psi_1$or $\psi_2$ is log-convex (log-concave), $F$ is IRFR(DRFR) distribution, $X \geq_{st} (\leq_{st}) W$ and $\phi_2\cdot\psi_1 (\phi_1\cdot\psi_2)$ is super additive.
\end{theorem}
\begin{proof} The survival function of $Y_{n:n} $ and $Z_{n:n}$ are 
	\begin{equation} 
	J_2(\underline{\mu};F(x),\psi_1) = 1-\psi_1\left(\displaystyle\sum_{k=1}^n \phi_1(F(x + \mu_k)) \right),
	\end{equation}
	\begin{equation}
	J_2(\underline{\mu};G(x),\psi_2) = 1-\psi_2\left(\displaystyle\sum_{k=1}^n \phi_2(G(x + \mu_k^{*})) \right)
	\end{equation}
	respectively.\\
It is sufficient to show that $\psi_1\left(\displaystyle\sum_{k=1}^n \phi_1(F(x + \mu_k)) \right) \leq(\geq) \psi_2\left(\displaystyle\sum_{k=1}^n \phi_2(G(x + \mu_k^{*})) \right)$
 Using Lemma \ref{lemmaj2} and Lemma \ref{lemma3} it can derived that \\$\underline{\mu} \prec_{w} \underline{\mu}^{*}~ (\underline{\mu} \prec^{w} \underline{\mu}^{*}) ~ \Rightarrow \psi_1\left(\displaystyle\sum_{k=1}^n \phi_1(F(x + \mu_k))\right) \leq(\geq) \psi_1\left(\displaystyle\sum_{k=1}^n \phi_1(F(x + \mu_k^{*}))\right)$.\\
 Since $\phi_2\cdot\psi_1 (\phi_1\cdot\psi_2)$ is super additive, we observe that 
 $$\psi_1\left(\displaystyle\sum_{k=1}^n \phi_1(F(x + \mu_k^{*}))\right) \leq(\geq) \psi_2\left(\displaystyle\sum_{k=1}^n \phi_2(F(x + \mu_k^{*}))\right)$$
 Lastly $X \geq_{st} (\leq_{st}) W$ implies that $\psi_2\left(\displaystyle\sum_{k=1}^n \phi_2(F(x + \mu_k^{*}))\right) \leq(\geq) \psi_2\left(\displaystyle\sum_{k=1}^n \phi_2(G(x + \mu_k^{*}))\right)$. Hence proved. \qed
\end{proof}
When we take the generator function $\psi(x)=\exp(-x)$, $\phi(x)=-\ln x$. This generator indicates the independence copula (when the random variables are independent). Subsequently one can obtain the usual stochastic ordering between two sets of independent random variables.\\

Consider the Clayton copula generator function as
$$\psi_{\theta}(x) = \max((1+\theta x)^{-1/{\theta}},0), ~\theta >~0.$$
The above Archimedean generator is completely monotone (n-monotone for every $n \in \mathbb{N}$) for $\theta >0$, and hence generates an Archiedean Copula.
Here $\psi_{\theta}$ is a log-convex function, and hence the above theorems hold for this archimedean generator.
 
\section{Conclusion}
The results discussed in this paper can be divided into 3 subparts as Proportional Hazard rate (PHR) model, Proportional Reversed Hazard rate (PRHR) model, Dependent model. For PHR model we considered different models, a generalized situation where we consider two sets of independent PHR random variables and the baseline distribution for both the sets are different ($X_1, X_2,\ldots,X_{n_1}$ such that $ X_i \sim \overline{F}_i(x) = (\overline{F}_{0}(x))^{\alpha_i}$ for $i=1, 2,\ldots, n_1$ and another set $Y_1,Y_2,\ldots,Y_{n_2}$ , $Y_i \sim \overline{G}_i(x) = (\overline{G}_{0}(x))^{\beta_i}$ for $i=1, 2,\ldots, n_2$ ). We have obtained conditions over the parameters and the baseline distributions so that a dispersive ordering exist between the minimum order statistics. Whereas when both the baseline distributions are same, star ordering occurs between these minimum order statistics provided $xr(x)$ is decreasing. Since Pareto distribution is also PHR model, a reversed hazard rate ordering occurs between the sample maximums (also known as parallel systems) when the shape parameter varies. Proceeding similarly we have observed a general result for PRHR model too. Here the two sets of random variables follow different baseline distributions and the number of samples are also unequal ($X_i \sim F_i(x) = (F_{0}(x))^{\alpha_i}$ for $i=1, 2,\ldots, n_1$ and $Y_i \sim G_i(x) = (G_{0}(x))^{\beta_i}$ for $i=1, 2,\ldots, n_2$). All of these results are true for multiple-outlier models.\\
Another form of generalized model has been studied where $X_1,\ldots,X_{p_1}$ has survival function $[\overline{F}(x)]^{\alpha_i}$ and $X_{p_1+1},\ldots,X_{n_1}$ has survival function $[\overline{G}
(x)]^{\alpha_i}$. Similarly another system with $n_2$ components is considered where the components $Y_1,\ldots,Y_{p_2}$ has survival function as $[\overline{F}(x)]^{\beta_i}$ whereas the components $Y_{p+1},\ldots,Y_{n_2}$ has survival function $[\overline{G}
(x)]^{\beta_i}$ and hazard rate ordering results has been observed for series systems. A reversed hazard rate ordering result with PRHR components has been observed.\\
In the last section, dependent random variables have been studied. Here we obtained usual stochastic ordering results between two sample minimums and two sample maximums such that the location parameter corresponding to the random variables from two sets obeys a weak majorization ordering while the baseline distribution obeys a usual stochastic ordering and the generating functions follows super-additive property.

\section{Acknowledgements}
The first author M.D. would like to thank IIT Kharagpur for research assistantship.



 \section*{Appendix}
\subsection{Useful results}

\begin{lemma}[Theorem 3.A.4, see {\cite{16}}]\label{lemma1}
	\normalfont Let $$ \Delta=(a_{i} - a_{j})\left(\dfrac{\partial\psi(\underline{a})}{\partial a_{i}} - \dfrac{\partial\psi(\underline{a})}{\partial a_{j}}\right),$$
	for an open interval $\mathbb{A} \subset \mathbb{R}$, a continuously differentiable function $\psi :\mathbb{A}^{n} \rightarrow \mathbb{R}$ is Schur-convex (Schur-concave) if and only if it is symmetric on $\mathbb{A}^{n}$ and for all $i \neq j$, $\Delta \geq(\leq)0.$
\end{lemma}
\begin{lemma}[Proposition 3.C.1, see {\cite{16}}]\label{lemma2}
	\normalfont If $\mathbb{A} \subset \mathbb{R}$ is an interval and $h: \mathbb{A} \rightarrow \mathbb{R}$ is convex (concave), then $\psi(\underline{a}) = \displaystyle\sum_{i=1}^n h(a_i)$ is Schur-convex (Schur-concave) on $\mathbb{A}^n$, where $\underline{a} = (a_1, \ldots, a_n)$.
\end{lemma}
\begin{lemma}[Theorem 3.A.8, see {\cite{16}}] \label{lemma3}
\normalfont	Let $S \subset \mathbb{R}^n$, a function $f: S \rightarrow \mathbb{R}$ satisfying $$\underline{a} \prec_w \underline{b} ~(\underline{a} \prec^w \underline{b}) \mbox{ on S } \Rightarrow f(\underline{a}) \leq f(\underline{b})$$ if and only if $f$ is increasing (decreasing) and Schur-convex on $S$.
\end{lemma}

\begin{lemma}[\cite{2}] \label{lemma4}
	\normalfont  Let {$F_{\alpha}$, $\alpha \in \mathbb{R}$} be a class of distribution functions such that the support of $F_{\alpha}$ is given by some interval $(x_0, x_1) \subset \mathbb{R^{+}}$ and has a non- vanishing density on any subinterval of $(x_0, x_1)$, where $x_0$ and  $x_1$ are the left and right end points respectively. Then
	\begin{equation} \label{l1}
	F_{\alpha} \leq_{disp} F_{\alpha^{*}}, ~\alpha,\alpha^{*} \in \mathbb{R}, ~ \alpha \leq \alpha^{*}, 
	\end{equation}
	if and only if $\dfrac{F^{'}_{\alpha}(x)}{f_{\alpha}(x)}$ is decreasing in $x$, where $F^{'}_{\alpha}$ is the derivative of $F_{\alpha}$ with respect to $\alpha$. \\
	And 
	\begin{equation} \label{l2}
	F_{\alpha} \leq_{*} F_{\alpha^{*}}, ~\alpha,\alpha^{*} \in \mathbb{R}, ~\alpha \leq \alpha^{*}, 
	\end{equation}
	if and only if $\dfrac{F^{'}_{\alpha}(x)}{x f_{\alpha}(x)}$ is decreasing in $x$, where $F^{'}_{\alpha}$ is the derivative of $F_{\alpha}$ with respect to $\alpha$. \\
	The first inequalities in \eqref{l1} and \eqref{l2} reverses as the quantity $\dfrac{F^{'}_{\alpha}(x)}{f_{\alpha}(x)}$ and $\dfrac{F^{'}_{\alpha}(x)}{x f_{\alpha}(x)}$ respectively increases in $x$.
\end{lemma}

\end{document}